\documentclass[amssymb, 11pt]{amsart}
\usepackage{latexsym}
\usepackage{hyperref}

\newlength{\standardunitlength}
\newcommand{\es}[1]{\begin{equation}\begin{split}#1\end{split}\end{equation}}

\newtheorem{prop}{Proposition}[section]

\newtheorem{lemma}[prop]{Lemma}
\newtheorem{cor}[prop]{Corollary}
\newtheorem{theorem}[prop]{Theorem}

\newcommand{\qbinom}[2]{{#1 \choose #2}_q}
\newcommand{\pbinom}[2]{{#1 \choose #2}_p}

\newcommand{\qfac}[1]{\left[#1\right]_q!}

\newcommand{\Z}{\mathbb{Z}}
\newcommand{\Q}{\mathbb{Q}}

\newcommand{\LL}{\mathcal{L}}

\newcommand{\GL}{{\operatorname{GL}}}

\newcommand{\Sp}{{\operatorname{Sp}}}
\newcommand{\Aut}{{\operatorname{Aut}}}
\newcommand{\Hom}{{\operatorname{Hom}}}

\newcommand{\Sym}{{\operatorname{Sym}}}
\newcommand{\Alt}{{\operatorname{Alt}}}
\newcommand{\M}{{\operatorname{M}}}

\newcommand{\Surj}{{\rm Sur}}

\title[Random Partitions and Cohen-Lenstra Heuristics]{Random Partitions and Cohen-Lenstra heuristics}
\author{Jason Fulman}
\address{Department of Mathematics, University of Southern California, Los Angeles, CA 90089}
\email{fulman@usc.edu}

\author{Nathan Kaplan}
\address{Department of Mathematics, University of California, Irvine, CA 92697-3875}
\email{nckaplan@math.uci.edu}

\thanks{{\it AMS classification numbers}: 15B52,05E05}

\thanks{Fulman is supported by Simons Foundation Grant 400528. Kaplan is supported by NSA Young
Investigator Grant H98230-16-10305, NSF Grant DMS 1802281 and by an AMS-Simons Travel Grant.  The authors thank the referees, Gilyoung Cheong, and Melanie Matchett Wood for helpful comments.}

\date{Version of January 24, 2019}


\begin{document}

\begin{abstract}
We investigate combinatorial properties of a family of probability distributions on finite abelian $p$-groups.  This family includes several well-known distributions as specializations.  These specializations have been studied in the context of Cohen-Lenstra heuristics and cokernels of families of random $p$-adic matrices.
\end{abstract}

\maketitle

\section{Introduction}

Friedman and Washington study a distribution on finite abelian p-groups $G$ of rank at
most $d$ in \cite{FW}.  In particular, a finite abelian $p$-group $G$ of rank $r \leq d$, is chosen with probability
\begin{equation} \label{FWform}
P_d(G) = \frac{1}{|\Aut(G)|} \left( \prod_{i=1}^d (1-1/p^i) \right) \left(\prod_{i=d-r+1}^d (1-1/p^i) \right).
\end{equation}
Let $\lambda = (\lambda_1,\ldots, \lambda_r)$ with $\lambda_1 \ge \lambda_2 \ge \cdots \ge \lambda_r \ge 1$ be a partition.  A finite abelian $p$-group $G$ has \emph{type} $\lambda$ if
\[
G \cong \Z/p^{\lambda_1}\Z \times \cdots \times \Z/p^{\lambda_r}\Z.
\]
Note that $r$ is equal to the rank of $G$.

There is a correspondence between measures on the set of integer partitions and on isomorphism classes of finite abelian $p$-groups.  Let $\LL$ denote the set of isomorphism classes of finite abelian $p$-groups.  Given a measure $\nu$ on partitions, we get a corresponding measure $\nu'$ on $\LL$ by setting $\nu'(G) = \nu(\lambda)$ where $G \in \LL$ is the isomorphism class of finite abelian $p$-groups of type $\lambda$.  We analogously define a measure on partitions given a measure on $\LL$.  When $G$ is a finite abelian group of type $\lambda$, we write $|\Aut(\lambda)|$ for $|\Aut(G)|$, and from page 181 of \cite{Mac},
\begin{equation}\label{AutDef}
|\Aut(\lambda)| = p^{\sum (\lambda_i')^2} \prod_i (1/p)_{m_i(\lambda)}.
\end{equation} The notation used in \eqref{AutDef} is standard, and we review it in Section \ref{secnot}.

We introduce and study a more general distribution on integer partitions and on finite abelian $p$-groups $G$ of rank at most $d$.  We choose a partition $\lambda$ with $r\le d$ parts with probability
\begin{equation} \label{gen}
P_{d,u}(\lambda) =
 \frac{u^{|\lambda|}}{p^{\sum (\lambda_i')^2} \prod_i (1/p)_{m_i(\lambda)}}
\prod_{i=1}^d (1-u/p^i) \prod_{i=d-r+1}^d (1-1/p^i).
\end{equation}
This gives a distribution on partitions for all real $p>1$ and $0<u<p$.  We can include $p$ as an additional parameter and write $P_{d,u}^p(\lambda)$.  For clarity, we will suppress this additional notation except in Section \ref{alter}.  When $p$ is prime, we can interpret \eqref{gen} as a distribution on $\LL$.  When $p$ is not prime it does not make sense to talk about automorphisms of a finite abelian $p$-group, but in this case we can take \eqref{AutDef} as the definition of $|\Aut(\lambda)|$.

The main goal of this paper is to investigate combinatorial properties of the family of distributions of \eqref{gen}.  We begin by noting six interesting specializations of this measure.
\begin{itemize}

\item Setting $u=1$ in $P_{d,u}$ recovers $P_d$.

\item We define a distribution $P_{\infty,u}$ by
\[
\lim_{d\to \infty} P_{d,u}(\lambda) = P_{\infty,u}(\lambda) = \frac{u^{|\lambda|}}{|\Aut(\lambda)|} \prod_{i \geq 1} (1-u/p^i) .
\]
It is not immediately clear that this limit defines a distribution on partitions, but this follows from the sentence after Proposition \ref{HLform}, from Theorem \ref{markov}, or from Theorem \ref{momgen}, taking $\mu$ to be the trivial partition.

For $0<u<1$, this probability measure arises by choosing a random non-negative integer $N$ with probability $P(N=n) =
(1-u) u^n$, and then looking at the $z-1$ piece of a random element of the finite group $\GL(N,p)$.
See \cite{Fu} for details.

\item Note that
\[
P_{\infty,1}(\lambda) =  \frac{1}{|\Aut(\lambda)|} \prod_{i \geq 1} (1-1/p^i).
\]
This is the measure on partitions corresponding to the usual Cohen-Lenstra measure on finite abelian $p$-groups \cite{CL}.  It also arises from studying the
$z-1$ piece of a random element of the finite group $\GL(d,p)$ in the $d \rightarrow \infty$ limit \cite{Fu},
or from studying the cokernel of a random $d \times d\ p$-adic matrix in the $d \rightarrow \infty$ limit
\cite{FW}.

\item Let $w$ be a positive integer and $\lambda$ a partition. The $w$-probability
of $\lambda$, denoted by $P_w(\lambda)$, is the probability that a finite abelian $p$-group of type $\lambda$ is obtained by the following three step random process:

\begin{itemize}
\item Choose randomly a $p$-group $H$ of type $\mu$ with respect to the measure $P_{\infty,1}(\mu)$.
\item Then choose $w$ elements $g_1,\cdots,g_w$ of $H$ uniformly at random.
\item Finally, output $H/ \langle g_1,\cdots,g_w \rangle$, where $\langle g_1,\cdots,g_w \rangle$
denotes the group generated by $g_1,\cdots,g_w$.
\end{itemize}

From Example 5.9 of Cohen and Lenstra \cite{CL}, it follows that $P_w(\lambda)$ is a special case of (\ref{gen}):
\begin{equation}\label{Pwdist}
P_w(\lambda) = P_{\infty,1/p^w}(\lambda).
\end{equation}

\item We now mention two analogues of Proposition 1 of \cite{FW} for rectangular matrices. Let $w$ be a non-negative integer.  Friedman and Washington do not discuss this explicitly, but using the same methods as in \cite{FW} one can show that taking the limit as $d \rightarrow \infty$ of the probability that a randomly chosen $d \times (d+w)$ matrix over $\Z_p$ has cokernel isomorphic to a finite abelian $p$-group of type $\lambda$ is given by $P_{\infty,1/p^w}(\lambda)$.  See the discussion above Proposition 2.3 of \cite{W3}.

Similarly, Tse considers rectangular matrices with more rows than columns and shows that $P_{\infty,1/p^w}(\lambda)$ is equal to the $d \rightarrow \infty$ probability that a randomly chosen $(d+w) \times d$ matrix over $\Z_p$ has cokernel isomorphic to $\Z_p^w \oplus G$, where $G$ is a finite
abelian $p$-group of type $\lambda$ \cite{Tse}.

\item In Section \ref{alter} we see that the measure on partitions studied by Bhargava, Kane, Lenstra, Poonen and Rains \cite{BKLPR}, arising from taking the cokernel of a random alternating  $p$-adic matrix is also a special case of $P_{d,u}$.  Taking a limit as the size of the matrix goes to infinity gives a  distribution consistent with heuristics of Delaunay for Tate-Shafarevich groups of elliptic curves defined over $\Q$ \cite{Del2}.

\end{itemize}

A few of these specializations have received extensive attention in prior work:
\begin{itemize}

\item When $p$ is an odd prime, Cohen and Lenstra conjecture that $P_{\infty,1}$ models the distribution of $p$-parts of class groups of imaginary quadratic fields and $P_{\infty,1/p}$ models the distribution of $p$-parts of class groups of real quadratic fields \cite{CL}.  Theorem 6.3 in \cite{CL} gives the probability that a group chosen from $P_{\infty,1/p^w}$ has given $p$-rank.  For any $n$ odd, they show that the average number of elements of order exactly $n$ of a group drawn from $P_{\infty,1}$ is $1$, and that this average for a group drawn from $P_{\infty,1/p}$ is $1/n$ \cite[Section 9]{CL}.  Delaunay generalizes these results in Corollary 11 of \cite{Del}, where he computes the probability that a group drawn from $P_{\infty,u}$ simultaneously has specified $p^j$-rank for several values of $j$. Delaunay and Jouhet compute averages of even more complicated functions involving moments of the number of $p^j$-torsion points for varying $j$ in \cite{DelJou}.

The distribution of $2$-parts of class groups of quadratic fields is not modeled by $P_{\infty,u}$ and several authors have worked to understand these issues.  Motivated by work of Gerth \cite{Ger,Ger2}, Fouvry and Kl\"uners study the conjectural distribution of $p^j$-ranks and moments for the number of torsion points of $C_D^2$, the square of the ideal class group of a quadratic field \cite{FoKlu}.

\item Delaunay \cite{Del}, and Delaunay and Jouhet \cite{DelJou}, prove analogues of the results described in the previous paragraphs for groups drawn from the $n \rightarrow \infty$ specialization of the distribution we study in Section \ref{alter}.  In \cite{DelJou2}, they prove analogues of the results of Fouvry and Kl\"uners \cite{FoKlu} for this distribution.

\end{itemize}

\subsection{Outline of the Paper}\label{outline}
In Section \ref{cok_padic} we interpret $P_{d,u}$ in terms of Hall-Littlewood polynomials and use this interpretation to compute the probability that a partition chosen from $P_{d,u}$ has given size, given number of parts, or given size and number of parts.  In Theorem \ref{markov} we give an algorithm for producing a partition according to the distribution $P_{d,u}$.

In Section \ref{alter} we show how a measure studied in \cite{BKLPR} that arises from distributions of cokernels of random alternating $p$-adic matrices is given by a specialization of $P_{d,u}$.  In Section \ref{sym} we briefly study a measure on partitions that arises from distributions of cokernels of random symmetric $p$-adic matrices that is studied in \cite{CKLPW, W}.  We give an algorithm for producing a partition according to this distribution.

In Section \ref{secmom} we combinatorially compute the moments of the distribution $P_{d,u}$ for all $d$ and $u$. These moments were already known for the case $d=\infty,u=1$, and our method is new even in that special case. We also show that in many cases these moments determine a unique distribution.  This is a generalization of a result of Ellenberg, Venkatesh, and Westerland \cite{EVW}, that the moments of the Cohen-Lenstra distribution determine the distribution, and of Wood \cite{W3}, that the moments of the distribution $P_w$ determine the distribution.

\subsection{Notation}\label{secnot}

Throughout this paper, when $p$ is a prime number we write $\Z_p$ for the ring of $p$-adic integers.

For a ring $R$, let $\M_d(R)$ denote the set of all $d \times d$ matrices with entries in $R$ and let $\Sym_d(R)$ denote the set of all $d \times d$ symmetric matrices with entries in $R$. For an even integer $d$, let $\Alt_d(R)$ denote the set of all $d \times d$ alternating matrices with entries in $R$ (that is, matrices $A$ with zeros on the diagonal satisfying that the transpose of $A$ is equal to $-A$).

For groups $G$ and $H$ we write $\Hom(G,H)$ for the set of homomorphisms from $G$ to $H,\ \Surj(G,H)$ for the set of surjective homomorphisms from $G$ to $H$, and $\Aut(G)$ for the set of automorphisms of $G$.  If $G$ is a finite abelian $p$-group of type $\lambda$ and $H$ is a finite abelian $p$-group of type $\mu$, we sometimes write $|\Surj(\lambda,\mu)|$ for $|\Surj(G,H)|$.

For a partition $\lambda$, we let $\lambda_i$ denote the size of the $i$\textsuperscript{th} part of $\lambda$ and $m_i(\lambda)$ denote the number of parts of $\lambda$ of size $i$.  We let $\lambda_i'$ denote the size of the $i$\textsuperscript{th} column in the diagram of $\lambda$ (so
$\lambda_i' = m_i(\lambda) + m_{i+1}(\lambda) + \cdots$). We also let
$n(\lambda)= \sum_i {\lambda_i' \choose 2}$. We generally use $r$ or $r(\lambda)$ to denote the number of parts of $\lambda$.  We use $|\lambda| = n$ to say that $\lambda$ is a partition of $n$, or equivalently $\sum \lambda_i = n$.

We let $n_{\lambda}(\mu)$ denote the number of subgroups of type $\mu$ of a finite abelian $p$-group of type $\lambda$.  For a finite abelian group $G$, the number of subgroups $H \subseteq G$ of type $\mu$ equals the number of subgroups for which $G/H$ has type $\mu$ \cite[Equation (1.5), page 181]{Mac}.

We also let $(x)_i = (1-x) (1-x/p) \cdots (1-x/p^{i-1})$.  So $(1/p)_i = (1-1/p) \cdots (1-1/p^i)$.  With this notation, \eqref{gen} is equivalent to
\[
P_{d,u}(\lambda) =
 \frac{u^{|\lambda|} (u/p)_d }{p^{\sum (\lambda_i')^2} \prod_i (1/p)_{m_i(\lambda)}}
 \frac{(1/p)_d}{(1/p)_{d-r(\lambda)}}.
\]

We use some notation related to $q$-binomial coefficients, namely:
\begin{eqnarray*}
[n]_q & = & \frac{q^n-1}{q-1} = 1 + q + \cdots + q^{n-1}; \\
\qfac{n} & = & [n]_q [n-1]_q \cdots [2]_q; \\
\qbinom{n}{j} & = &  \frac{\qfac{n}}{\qfac{j} \qfac{n-j}}.
\end{eqnarray*}

Finally if $f(u)$ is a power series in $u$, we let Coef. $u^n$ in $f(u)$ denote the coefficient of $u^n$
in $f(u)$.

\section{Properties of the measure $P_{d,u}$}\label{cok_padic}

To begin we give a formula for $P_{d,u}(\lambda)$ in terms of Hall-Littlewood polynomials. We let $P_{\lambda}$
denote a Hall-Littlewood polynomial, defined for a partition $\lambda=(\lambda_1,\cdots,\lambda_n)$ of length
at most $n$ by
\[
P_{\lambda}(x_1,\cdots,x_n;t) = \frac{1}{v_{\lambda}(t)} \sum_{w \in S_n}
w \left(x_1^{\lambda_1} \cdots x_n^{\lambda_n} \prod_{i<j} \frac{x_i-tx_j}{x_i-x_j} \right),
\]
where
\[
v_{\lambda}(t) = \prod_{i \geq 0} \prod_{j=1}^{m_i(\lambda)} \frac{1-t^j}{1-t},
\]
the permutation $w \in S_n$ permutes the $x$ variables, and we note that some parts of $\lambda$ may have size $0$.  For background on Hall-Littlewood polynomials, see Chapter 3 of \cite{Mac}.

\begin{prop} \label{HLform} For a partition $\lambda$ with $r \leq d$ parts,
\[
P_{d,u}(\lambda) = \prod_{i=1}^d (1-u/p^i) \cdot \frac{P_{\lambda}(\frac{u}{p},\frac{u}{p^2},\cdots,\frac{u}{p^d},0,\cdots ; \frac{1}{p})}{p^{n(\lambda)}} .
\]
\end{prop}

\begin{proof}
From page 213 of \cite{Mac},
\[
\prod_{i=1}^d (1-u/p^i) \cdot \frac{P_{\lambda}(\frac{u}{p},\frac{u}{p^2},\cdots,\frac{u}{p^d},0,\cdots ; \frac{1}{p})}{p^{n(\lambda)}}
\]
is equal to
\[
\frac{u^{|\lambda|} \prod_{i=1}^d (1-u/p^i)}{\prod_i (1/p)_{m_i(\lambda)}}
\frac{(1/p)_d}{p^{|\lambda|+2n(\lambda)} (1/p)_{d-r}}.
\]
Since $|\lambda|+2n(\lambda) = \sum (\lambda_i')^2$,
this is equal to \eqref{gen}, and the proposition follows.
\end{proof}

The fact that $\sum_{\lambda} P_{d,u}(\lambda) = 1$ follows from Proposition \ref{HLform}
and the identity of Example 1 on page 225 of \cite{Mac}. It is also immediate from Theorem
\ref{markov}.

There are two ways to generate random partitions $\lambda$ according to the distribution $P_{d,u}$.
The first is to run the ``Young tableau algorithm'' of \cite{Fu}, stopped when coin $d$ comes up
tails. The second method is given by the following theorem. 

\begin{theorem} \label{markov} Starting with $\lambda_0'=d$, define in succession $d \geq \lambda_1' \geq \lambda_2'
\geq \cdots$ according to the rule that if $\lambda_i'=a$, then $\lambda_{i+1}'=b$ with probability
\[
K(a,b) = \frac{u^b (1/p)_a (u/p)_a}{p^{b^2} (1/p)_{a-b} (1/p)_b (u/p)_b} .
\]
Then the resulting partition is
distributed according to $P_{d,u}$.
\end{theorem}

\begin{proof} One must compute
\[
K(d,\lambda_1') K(\lambda_1',\lambda_2') K(\lambda_2',\lambda_3') \cdots .
\]
There is a lot of
cancellation, and (recalling that $\lambda_1'=r$), what is left is:
\[
\frac{(u/p)_d (1/p)_d u^{|\lambda|}} {(1/p)_{d-r} p^{\sum (\lambda_i')^2} \prod_i (1/p)_{m_i(\lambda)}}.
\]
This is equal to $P_{d,u}(\lambda)$, completing the proof.
\end{proof}

The following corollary is immediate from Theorem \ref{markov}.

\begin{cor} \label{numparts} Choose $\lambda$ from $P_{d,u}$. Then the chance that $\lambda$ has $r\le d$ parts
is equal to
\[
\frac{u^r (1/p)_d (u/p)_d}{p^{r^2} (1/p)_{d-r} (1/p)_r (u/p)_r} .
\]
\end{cor}

\begin{proof} From Theorem \ref{markov}, the sought probability is $K(d,r)$. \end{proof}

The $u=1$ case of this result appears in another form in work of Stanley and Wang \cite{SW}.  In Theorem 4.14 of \cite{SW}, the authors compute the probability $Z_d(p,r)$ that the Smith normal form of a certain model of random integer matrix has at most $r$ diagonal entries divisible by $p$.  Setting $u=1$ in Corollary \ref{numparts} gives $Z_d(p,r) - Z_d(p,r-1)$.  This expression also appears in \cite{CKK} where the authors study finite abelian groups arising as $\Z^d/\Lambda$ for random sublattices $\Lambda \subset \Z^d$; isolating the prime $p$ and the $i = r$ term in Corollary 1.2 of \cite{CKK} gives the $u=1$ case of Corollary \ref{numparts}.

The next result computes the chance that $\lambda$ chosen from $P_{d,u}$ has size $n$.

\begin{theorem} \label{size}
The chance that $\lambda$ chosen from $P_{d,u}$ has size $n$ is equal to
\[ \frac{u^n}{p^n} \frac{(u/p)_d (1/p)_{d+n-1}}{(1/p)_{d-1} (1/p)_n}. \]
\end{theorem}

\begin{proof} By Proposition \ref{HLform}, the sought probability is equal to
\begin{eqnarray*}
\sum_{|\lambda|=n} P_{d,u}(\lambda) & = & (u/p)_d \sum_{|\lambda|=n} \frac{P_{\lambda}(\frac{u}{p},\frac{u}{p^2},\cdots,\frac{u}{p^d},0,\cdots ; \frac{1}{p})}{p^{n(\lambda)}} \\
& = & (u/p)_d \sum_{|\lambda|=n} u^n \frac{P_{\lambda}(\frac{1}{p},\frac{1}{p^2},\cdots,\frac{1}{p^d},0,\cdots ; \frac{1}{p})}{p^{n(\lambda)}} \\
& = & u^n (u/p)_d\ {\rm Coef. } u^n\ {\rm in}\ \sum_{\lambda} \frac{P_{\lambda}(\frac{u}{p},\frac{u}{p^2},\cdots,\frac{u}{p^d},0,\cdots ; \frac{1}{p})}{p^{n(\lambda)}} \\
& = & u^n (u/p)_d\ {\rm Coef.} u^n\ {\rm in}\ \frac{1}{(u/p)_d} \\
& = & \frac{u^n}{p^n} \frac{(u/p)_d (1/p)_{d+n-1}}{(1/p)_{d-1} (1/p)_n}.
\end{eqnarray*}
The fourth equality used Proposition \ref{HLform} and the fact that $P_{d,u}$ defines a probability distribution, and the final equality
used Theorem 349 of \cite{HW}.
\end{proof}


\begin{theorem} \label{parts_and_size}
The probability that $\lambda$ chosen from $P_{d,u}$ has size $n$ and $r \le \min\{d,n\}$ parts is equal
to \[ \frac{u^n (u/p)_d (1/p)_d}{p^{r^2} (1/p)_{d-r} (1/p)_r} \frac{(1/p)_{n-1}}
{p^{n-r} (1/p)_{r-1} (1/p)_{n-r}} .\]
\end{theorem}

\begin{proof} From the definition of $P_{d,u}$, one has that
\begin{eqnarray*}
\sum_{\substack{\lambda_1'=r \\ |\lambda|=n}} P_{d,u}(\lambda) & = & \sum_{\substack{\lambda_1'=r \\ |\lambda|=n}}
\frac{u^n (u/p)_d (1/p)_d}{|\Aut(\lambda)| (1/p)_{d-r}} \\
& = & u^n (u/p)_d \sum_{\substack{\lambda_1'=r \\ |\lambda|=n} }\frac{(1/p)_d}{|\Aut(\lambda)| (1/p)_{d-r}} \\
& = & u^n (u/p)_d\ {\rm Coef.}\ u^n\ {\rm in}\ \sum_{\lambda_1'=r} \frac{u^{|\lambda|} (1/p)_d}{|\Aut(\lambda)| (1/p)_{d-r}} \\
& = & u^n (u/p)_d\ {\rm Coef.}\ u^n\ {\rm in}\ \frac{1}{(u/p)_d} \sum_{\lambda_1'=r} P_{d,u}(\lambda) \\
& = & u^n (u/p)_d\ {\rm Coef.}\ u^n\ {\rm in}\ \frac{1}{(u/p)_d} \frac{u^r (1/p)_d (u/p)_d}{p^{r^2} (1/p)_{d-r} (1/p)_r (u/p)_r} \\
& = & \frac{u^n (u/p)_d (1/p)_d}{p^{r^2} (1/p)_{d-r} (1/p)_r}\ {\rm Coef.}\ u^{n-r}\ {\rm in}\ \frac{1}{(u/p)_r}\\
& = & \frac{u^n (u/p)_d (1/p)_d}{p^{r^2} (1/p)_{d-r} (1/p)_r} \frac{(1/p)_{n-1}}
{p^{n-r} (1/p)_{r-1} (1/p)_{n-r}}.
\end{eqnarray*} The fifth equality used Corollary \ref{numparts}, and the final equality used Theorem
349 of \cite{HW}.
\end{proof}

In the rest of this section we give another view of the distributions given by \eqref{FWform} and
\eqref{gen}. When $p$ is prime, equation (19) in \cite{petrogradsky} implies that
\begin{equation}
P_d(\lambda)  = \frac{1}{p^{|\lambda| d}}
\left(\prod_{i=1}^{\lambda_1} p^{\lambda_{i+1}' (d- \lambda_i')} \pbinom{d-\lambda_{i+1}'} {\lambda_i' - \lambda_{i+1}'}\right)  \prod_{i=1}^{d} (1-1/p^i).
\end{equation}
Comparing this to the expression for $P_d(\lambda)$ given in \eqref{FWform} shows that
\begin{equation}\label{petro_FW}
\frac{1}{p^{|\lambda| d}}
\left(\prod_{i=1}^{\lambda_1} p^{\lambda_{i+1}' (d- \lambda_i')} \pbinom{d-\lambda_{i+1}'} {\lambda_i' - \lambda_{i+1}'}\right)
=
\frac{1}{|\Aut(\lambda)|} \left(\prod_{i=d-r+1}^d (1-1/p^i) \right).
\end{equation}
A direct proof is given in Proposition 4.7 of \cite{CKK}.  Therefore, we get a second expression for $P_{d,u}(\lambda)$,
\begin{equation}\label{petro_pdu}
P_{d,u}(\lambda)  =  \frac{u^{|\lambda|}}{p^{|\lambda| d}}
\left(\prod_{i=1}^{\lambda_1} p^{\lambda_{i+1}' (d- \lambda_i')} \pbinom{d-\lambda_{i+1}'} {\lambda_i' - \lambda_{i+1}'}\right)  \prod_{i=1}^{d} (1-u/p^i).
\end{equation}
We give a combinatorial proof of \eqref{petro_FW} that applies for any real $p > 1$, so \eqref{petro_pdu} applies for any $p>1$ and $0<u<p$.
\begin{proof}[Proof of Equation \eqref{petro_FW}]
It is sufficient to show that for a partition $\lambda$ with $r \leq d$ parts,
\begin{equation} \label{mustshow} |\Aut(\lambda)| \left( \prod_{i=1}^{\lambda_1} p^{\lambda_{i+1}'(d-\lambda_i')} {d-\lambda_{i+1}' \choose \lambda_i' - \lambda_{i+1}'}_p \right) = p^{|\lambda| d} \prod_{j=0}^{r-1} (1-p^{-d+j}).
\end{equation}

Clearly
\begin{eqnarray*}
& & \prod_{i=1}^{\lambda_1} p^{\lambda_{i+1}'(d-\lambda_i')}  {d-\lambda_{i+1}' \choose \lambda_i' - \lambda_{i+1}'}_p \\
& = &
p^{d(|\lambda| - \lambda_1') - \sum_i \lambda_i' \lambda_{i+1}'} \prod_i {d-\lambda_{i+1}' \choose \lambda_i' - \lambda_{i+1}'}_p \\
& = & p^{d(|\lambda| - \lambda_1') - \sum_i \lambda_i' \lambda_{i+1}'} \frac{[d]_p!}{[d-\lambda_1']_p! [\lambda_1'-\lambda_2']_p!
[\lambda_2'-\lambda_3']_p! \cdots} \\
& = & p^{d(|\lambda| - \lambda_1') - \sum_i \lambda_i' \lambda_{i+1}'} \frac{(p-1)^{\lambda_1'} [d]_p!}
{[d-\lambda_1']_p! p^{\sum_i {\lambda_i'-\lambda_{i+1}'+1 \choose 2}} \prod_i (1/p)_{m_i(\lambda)}} \\
& = & \frac{p^{d(|\lambda|-\lambda_1')} (p-1)^{\lambda_1'} [d]_p!}{[d-\lambda_1']_p! p^{\frac{1}{2} [\sum_i (\lambda_i')^2
+(\lambda_{i+1}')^2+\lambda_i'-\lambda_{i+1}']} \prod_i (1/p)_{m_i(\lambda)}} \\
& = & \frac{p^{d(|\lambda|-\lambda_1')} p^{(\lambda_1')^2/2} (p-1)^{\lambda_1'} [d]_p!}{[d - \lambda_1']_p! p^{\lambda_1'/2}}
\cdot \frac{1}{p^{\sum_i (\lambda_i')^2} \prod_i (1/p)_{m_i(\lambda)}}.
\end{eqnarray*}

Since $\lambda_1'=r$, equation \eqref{AutDef} implies that the left-hand side of (\ref{mustshow}) is equal to
\begin{eqnarray*}
& & \frac{p^{d |\lambda| - dr +r^2/2-r/2} (p-1)^r [d]_p!}{[d-r]_p!} \\
& = & p^{d |\lambda| - dr +r^2/2-r/2} (p^d-1) \cdots (p^{d-r+1}-1),
\end{eqnarray*}
which simplifies to the right-hand side of (\ref{mustshow}).
\end{proof}

We now use the alternate expression of \eqref{petro_pdu} to give an additional proof of Theorem \ref{size} in the case when $p$ is prime.  The zeta function of $\Z^d$ is defined by
\[
\zeta_{\Z^d}(s) = \sum_{H \le \Z^d} [\Z^d:H]^{-s},
\]
where the sum is taken over all finite index subgroups of $\Z^d$.  It is known that
\es{\label{zeta}
\zeta_{\Z^d}(s)  = & \zeta(s) \zeta(s-1)\cdots \zeta(s-(d-1)) \\
 = & \prod_p \left((1-p^{-s})^{-1} (1-p^{-(s-1)})^{-1}  \cdots (1-p^{-(s-(d-1))})^{-1}\right),
}
where $\zeta(s)$ denotes the Riemann zeta function, and the product is taken over all primes.  See the book of Lubotzky and Segal for five proofs
of this fact \cite{lubotzky_segal}.

\begin{proof}[Second Proof of Theorem \ref{size} for $p$ prime]
From \eqref{petro_pdu}, we need only prove
\begin{equation}\label{secondproofeq}
\sum_{|\lambda| =  n}  \frac{u^n}{p^{n d}}
\left(\prod_{i=1}^{\lambda_1} p^{\lambda_{i+1}' (d- \lambda_i')} \pbinom{d-\lambda_{i+1}'} {\lambda_i' - \lambda_{i+1}'}\right) = \frac{u^n}{p^n} \frac{(1/p)_{d+n-1}}{(1/p)_{d-1} (1/p)_n}.
\end{equation}

Let $\lambda^* = (\lambda_1,\ldots, \lambda_1)$, where there are $d$ entries in the tuple.  The discussion around equation (19) in \cite{petrogradsky} says that the term in parentheses of the left-hand side of \eqref{secondproofeq} is equal to the number of subgroups of a finite abelian $p$-group of type $\lambda^*$ that have type $\lambda,\ n_{\lambda^*}(\lambda)$, which is also equal to the number of subgroups $\Lambda \subset \Z^d$ such that $\Z^d/\Lambda$ is a finite abelian $p$-group of type $\lambda$.

After some obvious cancelation, we need only show that
\[
\sum_{|\lambda| = n} n_{\lambda^*}(\lambda) = \frac{p^{n(d-1)} (1/p)_{d+n-1}}{(1/p)_{d-1} (1/p)_n}.
\]
The left-hand side is the number of subgroups $\Lambda \subset \Z^d$ such that $\Z^d/\Lambda$ has order $p^n$.  This is the $p^{-sn}$ coefficient of $\zeta_{\Z^d}(s)$. Using \eqref{zeta}, this is equal to
\begin{eqnarray*}
& &  \text{Coef. $p^{-sn}$ in \ }(1-p^{-s})^{-1} (1-p^{-(s-1)})^{-1}  \cdots (1-p^{-(s-(d-1))})^{-1} \\
& = & \text{Coef. $x^n$ in \ }(1-x)^{-1} (1-px)^{-1} (1-p^2 x)^{-1}  \cdots (1-p^{d-1} x))^{-1}.
\end{eqnarray*} By Theorem 349 of \cite{HW}, this is equal to
\[ \frac{p^{n(d-1)} (1/p)_{d+n-1}}{(1/p)_{d-1} (1/p)_n}, \] and the proof is complete.

\end{proof}

\section{Cokernels of random alternating $p$-adic matrices} \label{alter}

In this section we consider a distribution on finite abelian $p$-groups that arises in the study of cokernels of random $p$-adic alternating matrices.  We show that this is a special case of the distributions $P^p_{d,u}$.

Let $n$ be an even positive integer and let $A\in \Alt_n(\Z_p)$ be a random matrix chosen with respect to additive Haar measure on $\Alt_n(\Z_p)$.  The cokernel of $A$ is a finite abelian $p$-group of the form $G \cong H \times H$ for some $H$ of type $\lambda$ with at most $n/2$ parts, and is equipped with a nondegenerate alternating pairing $[\ ,\ ]\colon H \times H \mapsto \Q/\Z$.  Let $\Sp(G)$ be the group of automorphisms of $H$ respecting $[\ ,\ ]$.  Let $r$ be the number of parts of $\lambda$, and $|\lambda|$, $n(\lambda)$,
$m_i(\lambda)$ be as in Section \ref{secnot}.
\begin{lemma}\label{altlemma}
Let $n$ be an even positive integer and $A\in \Alt_n(\Z_p)$ be a random matrix chosen with respect to additive Haar measure on $\Alt_n(\Z_p)$.  The probability that the cokernel of $A$ is isomorphic to $G$ is given by
\begin{equation} \label{altmeas}
P^{\text{Alt}}_{n,p}(\lambda) = \frac{\prod_{i=n-2r+1}^n (1-1/p^i) \prod_{i=1}^{n/2-r} (1-1/p^{2i-1})}
{p^{|\lambda|+4n(\lambda)} \prod_i \prod_{j=1}^{m_i(\lambda)} (1-1/p^{2j})}.
\end{equation}
\end{lemma}

\begin{proof}
Formula (6) and Lemma 3.6 of \cite{BKLPR} imply that the probability that the cokernel of $A$ is isomorphic to $G$ is given by
\[
\frac{\left|\Surj(\Z_p^n,G)\right|}{|\Sp(G)|} \prod_{i=1}^{n/2-r} (1-1/p^{2i-1}) |G|^{1-n}.
\]

We can rewrite this expression in terms of the partition $\lambda$. Clearly $|G|=p^{2 |\lambda|}$.  Proposition 3.1 of \cite{CL} implies that since $G$ has rank $2r$,
\[
|{\Surj}(\Z_p^n,G)| = p^{2n |\lambda|} \prod_{i=n-2r+1}^n (1-1/p^i).
\]

An identity on the bottom of page 538 of \cite{Del} says that,
\begin{eqnarray*}
|\Sp(G)| & = & p^{|\lambda|} p^{2 \sum_i (\lambda_i')^2} \prod_i \prod_{j=1}^{m_i(\lambda)} (1-1/p^{2j}) \\
& = & p^{4n(\lambda)+3|\lambda|} \prod_i \prod_{j=1}^{m_i(\lambda)} (1-1/p^{2j}).
\end{eqnarray*}
Putting these results together completes the proof.
\end{proof}

The next theorem shows that (\ref{altmeas}) is a special case of (\ref{gen}).

\begin{theorem}\label{AlternatingSpecialization}
Let $n$ be an even positive integer.  For any partition $\lambda$,
\[
P^{p^2}_{n/2,p}(\lambda) = P^{\text{Alt}}_{n,p}(\lambda).
\]

\end{theorem}

\begin{proof}
Rewrite (\ref{gen}) as
\[ \frac{u^{|\lambda|} \prod_{i=1}^d (1-u/p^i) \prod_{i=d-r+1}^d (1-1/p^i)}
{p^{2n(\lambda)+|\lambda|} \prod_i \prod_{j=1}^{m_i(\lambda)} (1-1/p^j)}.\]
Replacing $d$ by $n/2$, $u$ by $p$, and $p$ by $p^2$ gives
\[ \frac{\prod_{i=1}^{n/2} (1-1/p^{2i-1}) \prod_{i=n/2-r+1}^{n/2} (1-1/p^{2i})}
{p^{4 n(\lambda)+|\lambda|} \prod_i \prod_{j=1}^{m_i(\lambda)} (1-1/p^{2j})}.  \]

Comparing with (\ref{altmeas}), we see that it suffices to prove
\[
\prod_{i=1}^{n/2} (1-1/p^{2i-1}) \prod_{i=n/2-r+1}^{n/2} (1-1/p^{2i})
= \prod_{i=n-2r+1}^n (1-1/p^i) \prod_{i=1}^{n/2-r} (1-1/p^{2i-1}).
\]
To prove this equality, note that when each side
is multiplied by
\[
(1-1/p^2)(1-1/p^4) \cdots (1-1/p^{n-2r}),
\]
each side becomes $(1/p)_n$.
\end{proof}

\section{Cokernels of random symmetric $p$-adic matrices}\label{sym}

Let $A\in \Sym_n(\Z_p)$ be a random matrix chosen with respect to additive Haar measure on $\Sym_n(\Z_p)$.  Let $r$ be the number of parts of $\lambda$.  Theorem 2 of \cite{CKLPW} shows that the probability that the cokernel of $A$ has type $\lambda$ is equal to:
\begin{equation} \label{symform}
P^{\text{Sym}}_{n}(\lambda)  = \frac{\prod_{j=n-r+1}^n (1-1/p^j) \prod_{i=1}^{\lceil (n-r)/2 \rceil} (1-1/p^{2i-1})}{p^{n(\lambda)+|\lambda|} \prod_{i \geq 1} \prod_{j=1}^{\lfloor m_i(\lambda)/2 \rfloor} (1-1/p^{2j})}.
\end{equation}
Note that $P^{\text{Sym}}_{n}(\lambda)=0$  if $\lambda$ has more than $n$ parts.  As in earlier sections, when $p$ is prime \eqref{symform} has an interpretation in terms of finite abelian $p$-groups, but defines a distribution on partitions for any $p>1$. This follows directly from Theorem \ref{markovsym} below.

Taking $n \to \infty$ gives a distribution on partitions where $\lambda$ is chosen with probability
\begin{equation} \label{symforminf}
P^{\text{Sym}}_{\infty}(\lambda)  = \frac{\prod_{i \text{ odd }} (1-1/p^i)}{p^{n(\lambda)+|\lambda|} \prod_{i \geq 1} \prod_{j=1}^{\lfloor m_i(\lambda)/2 \rfloor} (1-1/p^{2j})}.
\end{equation}
The distribution of \eqref{symforminf} is studied in \cite{W}, where Wood shows that it arises as the distribution of $p$-parts of sandpile groups of large Erd\H{o}s-R\'enyi random graphs. Combinatorial properties of this distribution are considered in \cite{Fu3}, where it is shown that this distribution is a specialization of a two parameter family of distributions.  It is unclear whether the distribution of \eqref{symform} also sits within a larger family.

The following theorem allows one to generate partitions from the measure (\ref{symform}), and is a minor
variation on Theorem 3.1 of \cite{Fu3}.

\begin{theorem} \label{markovsym} Starting with $\lambda_0'=n$, define in succession $n \geq \lambda_1' \geq \lambda_2' \geq \cdots$ according to the rule that if $\lambda_l'=a$, then $\lambda_{l+1}'=b$ with probability
\[ K(a,b) = \frac{\prod_{i=1}^a (1-1/p^i)}{p^{{b+1 \choose 2}} \prod_{i=1}^b (1-1/p^i) \prod_{j=1}^{\lfloor (a-b)/2
\rfloor} (1-1/p^{2j})} .\] Then the resulting partition with at most $n$ parts is distributed according
to (\ref{symform}). \end{theorem}

\begin{proof} It is necessary to compute
\[ K(n,\lambda_1') K(\lambda_1',\lambda_2') K(\lambda_2',\lambda_3') \cdots \]
There is a lot of cancelation, and (recalling that $\lambda_1'=r$), what is left is:
\[
\frac{\prod_{j=1}^n (1-1/p^j)}{\prod_{j=1}^{\lfloor (n-r)/2 \rfloor} (1-1/p^{2j})}
 \frac{1}{p^{n(\lambda)+|\lambda|} \prod_{i \geq 1}
\prod_{j=1}^{\lfloor m_i(\lambda)/2 \rfloor} (1-1/p^{2j})}. \]

So to complete the proof, it is necessary to check that
\[ \frac{\prod_{j=1}^n (1-1/p^j)}{\prod_{j=1}^{\lfloor (n-r)/2 \rfloor} (1-1/p^{2j})}
= \prod_{j=n-r+1}^n (1-1/p^j) \prod_{i=1}^{\lceil (n-r)/2 \rceil} (1-1/p^{2i-1}). \]
This equation is easily verified by breaking it into cases based on whether $n-r$ is even
or odd.
\end{proof}

The following corollary is immediate.

\begin{cor} Let $\lambda$ be chosen from (\ref{symform}). Then the chance that $\lambda$ has
$r \leq n$ parts is equal to \[ \frac{\prod_{j=r+1}^n (1-1/p^j)}{p^{{r+1 \choose 2}}
\prod_{j=1}^{\lfloor (n-r)/2 \rfloor} (1-1/p^{2j})} .\]
\end{cor}

\begin{proof} By Theorem \ref{markovsym}, the sought probability is equal to $K(n,r)$.
\end{proof}
Taking $n \to \infty$ in this result recovers Theorem 2.2 of \cite{Fu3}, which is also Corollary 9.4 of \cite{W}.

\section{Computation of $H$-moments}\label{secmom}

We recall that $\LL$ denotes the set of isomorphism classes of finite abelian $p$-groups and that a probability distribution $\nu$ on $\LL$ gives a probability distribution on the set of partitions in an obvious way.  Similarly, a measure on partitions gives a measure on $\LL$, setting $\nu(G) = \nu(\lambda)$ when $G$ is a finite abelian $p$-group of type $\lambda$.  When $G, H \in \LL$ we write $|\Surj(G,H)|$ for the number of surjections from any representative of the isomorphism class $G$ to any representative of the isomorphism class $H$.

Let $\nu$ be a probability measure on $\LL$.  For $H \in \LL$, the $H$-moment of $\nu$ is defined as
\[
\sum_{G\in \LL} \nu(G)  |\Surj(G,H)|.
\]
When $H$ is a finite abelian $p$-group of type $\mu$ this is
\[
\sum_{\lambda} \nu(\lambda)  |\Surj(\lambda,\mu)|.
\]
The distribution $\nu$ gives a measure on partitions and we refer to this quantity as the $\mu$-moment of the measure.  For an explanation of why these are called the moments of the distribution, see Section 3.3 of \cite{CKLPW}.

The Cohen-Lenstra distribution is the probability distribution on $\LL$ for which a finite abelian group $G$ of type $\lambda$ is chosen with probability $P_{\infty,1}(\lambda)$.  One of the most striking properties of the Cohen-Lenstra distribution is that the $H$-moment of $P_{\infty,1}$ is $1$ for every $H$, or equivalently, for any finite abelian $p$-group $H$ of type $\mu$,
\[
\sum_{\lambda} P_{\infty,1}(\lambda)  |\Surj(\lambda,\mu)| = 1.
\]
There is a nice algebraic explanation of this fact using the interpretation of $P_{\infty,1}$ as a limit of the $P_{d,1}$ distributions given by \eqref{FWform} (see for example \cite{Speyer}).

Lemma 8.2 of \cite{EVW} shows that the Cohen-Lenstra distribution is determined by its moments.
\begin{lemma}\label{MomentsDetermine}
Let $p$ be an odd prime.  If $\nu$ is any probability measure on $\LL$ for which
\[
\sum_{G \in \LL} \nu(G)  |\Surj(G,H)| = 1
\]
for any $H \in \LL$, then $\nu = P_{\infty,1}$.
\end{lemma}

Our next goal is to compute the moments for the measure $P_{d,u}$; see Theorem \ref{momgen} below.
Our method is new even in the case $P_{\infty,1}$.

There has been much recent interest in studying moments of distributions related to the Cohen-Lenstra distribution, and showing that these moments determine a unique distribution \cite{BW, W, W3}.  At the end of this section, we add to this discussion by proving a version of Lemma \ref{MomentsDetermine} for the distribution $P_{d,u}$.

The following lemma counts the number of surjections from $G$ to $H$. Recall that $n_{\lambda}(\mu)$ is the
number of subgroups of type $\mu$ of a finite abelian group of type $\lambda$.

\begin{lemma} \label{nathan}
Let $G,H$ be finite abelian $p$-groups of types $\lambda$ and $\mu$
respectively. Then
\[
|\Surj(G,H)| = |\Surj(\lambda,\mu)|=  n_{\lambda}(\mu)
|\Aut(\mu)|.
\]
\end{lemma}
For a proof, see page 28 of \cite{Wright}.  The main idea is that $|\Surj(G,H)|$ is the number of \emph{injective} homomorphisms from $\widehat{H}$ to $\widehat{G}$, where these are the dual groups of $H$ and $G$, respectively.  The image of such a homomorphism is a subgroup of $\widehat{G}$ of type $\mu$.

The distributions $P_{d,u}$ are defined for all $p>1$.  It is not immediately clear what the $\mu$-moment of this distribution should mean when $p$ is not prime, since $|\Surj(\lambda,\mu)|$ is defined in terms of surjective homomorphisms between finite abelian $p$-groups.  In \eqref{AutDef} we saw how to define $|\Aut(\lambda)|$ in terms of the parts of the partition $\lambda$ and the parameter $p$, even in the case where $p$ is not prime.  Similarly, Lemma \ref{nathan} gives a way to define $|\Surj(\lambda, \mu)|$ in terms of the parameter $p$ and the partitions $\lambda$ and $\mu$ even when $p$ is not prime.  We first define $|\Aut(\mu)|$ using \eqref{AutDef}, and then note that $n_{\lambda}(\mu)$ is a polynomial in $p$ that we can evaluate for any $p>1$.

\begin{theorem} \label{momgen}
The $\mu$-moment of the distribution $P_{d,u}$ is equal to
\begin{equation*}
\left\{ \begin{array}{ll}
\frac{u^{|\mu|} (1/p)_d}{(1/p)_{d-r(\mu)}} & \mbox{if $r(\mu) \leq d$} \\
0 & \mbox{otherwise}.
\end{array} \right.
\end{equation*}
Here, as above, $r(\mu)$ denotes the number of parts of $\mu$.
\end{theorem}

\begin{proof} Clearly we can suppose that $r(\mu) \leq d$.
By Lemma \ref{nathan}, the $\mu$-moment of the distribution $P_{d,u}$ is equal to
\[
\sum_{\lambda} P_{d,u}(\lambda) |\Surj(\lambda,\mu)| = |\Aut(\mu)| \sum_{\lambda}
P_{d,u}(\lambda)n_{\lambda}(\mu) .
\]

Let $n_{\lambda}(\mu,\nu)$ be the number of subgroups $M$ of $G$ so that $M$ has type $\mu$ and
$G/M$ has type $\nu$. This is a polynomial in $p$ (see Chapter II Section 4 of \cite{Mac}). Then by
Proposition \ref{HLform}, the $\mu$-moment becomes
\[
|\Aut(\mu)| \prod_{i=1}^d (1-u/p^i)\cdot \sum_{\lambda} \frac{P_{\lambda}(\frac{u}{p},\frac{u}{p^2},
\cdots,\frac{u}{p^d},0,\cdots;\frac{1}{p})}{p^{n(\lambda)}} \sum_{\nu} n_{\lambda}(\mu,\nu).
\]
Reversing the order of summation, this becomes
\[
|\Aut(\mu)| \prod_{i=1}^d (1-u/p^i) \cdot \sum_{\nu} \sum_{\lambda} \frac{P_{\lambda}(\frac{u}{p},\frac{u}{p^2},
\cdots,\frac{u}{p^d},0,\cdots;\frac{1}{p})}{p^{n(\lambda)}} n_{\lambda}(\mu,\nu).
\]

From Section 3.3 of \cite{Mac}, it follows that for any values of the $x$ variables,
\[ \sum_{\lambda} n_{\lambda}(\mu,\nu) \frac{P_{\lambda}(x;\frac{1}{p})}{p^{n(\lambda)}}
= \frac{P_{\mu}(x;\frac{1}{p})}{p^{n(\mu)}} \frac{P_{\nu}(x;\frac{1}{p})}{p^{n(\nu)}}.\]
Specializing $x_i=u/p^i$ for $i=1,\cdots,d$ and $0$ otherwise, it follows that the $\mu$-moment
of $P_{d,u}$ is equal to
\begin{eqnarray*}
& & |\Aut(\mu)| \prod_{i=1}^d (1-u/p^i)\cdot  \sum_{\nu} \frac{P_{\mu}(\frac{u}{p},\frac{u}{p^2},
\cdots,\frac{u}{p^d},0,\cdots;\frac{1}{p})}{p^{n(\mu)}} \\
& & \cdot \frac{P_{\nu}(\frac{u}{p},\frac{u}{p^2},
\cdots,\frac{u}{p^d},0,\cdots;\frac{1}{p})}{p^{n(\nu)}} \\
& = & |\Aut(\mu)| \frac{P_{\mu}(\frac{u}{p},\frac{u}{p^2},
\cdots,\frac{u}{p^d},0,\cdots;\frac{1}{p})}{p^{n(\mu)}} \\
& & \cdot \sum_{\nu} \prod_{i=1}^d (1-u/p^i) \cdot \frac{P_{\nu}(\frac{u}{p},\frac{u}{p^2},
\cdots,\frac{u}{p^d},0,\cdots;\frac{1}{p})}{p^{n(\nu)}}.
\end{eqnarray*}

By Proposition \ref{HLform}, this is equal to
\[
|\Aut(\mu)| \frac{P_{\mu}(\frac{u}{p},\frac{u}{p^2},
\cdots,\frac{u}{p^d},0,\cdots;\frac{1}{p})}{p^{n(\mu)}}.
\]
By pages 181 and 213 of \cite{Mac},
this simplifies to
\[ \frac{u^{|\mu|} (1/p)_d}{(1/p)_{d-r(\mu)}}.  \]
\end{proof}

{\it Remarks:}
\begin{itemize}
\item The exact same argument proves the analogous result for the distribution $P_{\infty, u}$.
\item Setting $d=\infty$ and $u = 1/p^w$ (with $w$ a positive integer)
gives the distribution \eqref{Pwdist}, and in this case Theorem \ref{momgen} recovers Lemma 3.2 of \cite{W2}.

\item The argument used in the proof of Theorem \ref{momgen} does not
require that $p$ is prime.
\end{itemize}

We use Theorem \ref{momgen} to determine the expected number of $p^{\ell}$-torsion elements of a finite abelian group $H$ drawn from $P_{d,u}$.  Let $T_{\ell}$ be defined by
\[
T_{\ell}(H) = |H[p^{\ell}]| = |\{x\in H\colon p^\ell \cdot x = 0\}|.
\]
The number of elements of $H$ of order exactly $p^{\ell}$ is $T_\ell(H) - T_{\ell-1}(H)$.

For a finite abelian $p$-group $H$, let $r_{p^k}(H)$ denote the $p^k$-rank of $H$, that is,
\[
r_{p^k}(H) = \dim_{\Z/p\Z}\left(p^{k-1}H/p^k H\right).
\]
If $H$ is of type $\lambda$, then $r_{p^k}(H) = \lambda_k'$, the number of parts of $\lambda$ of size at least $k$.  The number of parts of $\lambda$ of size exactly $k$ is $\lambda_{k}' - \lambda_{k+1}'$.  It is clear that
\[
T_{\ell}(H) = p^{r_p(H) + r_{p^2}(H) + \cdots + r_{p^\ell}(H)} = p^{\lambda_1'+\lambda_2' + \cdots + \lambda_\ell'}.
\]

\begin{theorem}\label{torspts}
Let $p$ be a prime, $\ell$ be a positive integer, and $0 < u < p$. The expected value of $T_{\ell}(H)$ for a finite abelian $p$-group $H$ drawn from $P_{d,u}$ is
\[
(u^{\ell}+u^{\ell-1} + \cdots + u)(1-p^{-d}) + 1.
\]
The expected value of $T_{\ell}(H) - T_{\ell-1}(H)$ is $u^{\ell}(1-p^{-d})$.
\end{theorem}

{\it Remarks:}
\begin{itemize}
\item The exact same argument proves the analogous result for the distribution $P_{\infty, u}$.

\item Taking $d = \infty,\ u = p^{-w}$ recovers a result of Delaunay, the first part of Corollary 3 of \cite{Del}.  Delaunay's result generalizes work of Cohen and Lenstra for $P_{\infty,1}$ and $P_{\infty,1/p}$ \cite{CL}.

 \item Theorem \ref{momgen} can likely be used to compute moments of more complicated functions involving $T_{\ell}(H)$ giving results similar to those of Delaunay and Jouhet in \cite{DelJou}. We do not pursue this further here.
 \end{itemize}

\begin{lemma}\label{counthoms}
Let $H$ be a finite abelian $p$-group of type $\lambda$ and let $\ell \ge 1$.  Then
\[
\#\Hom(H,\Z/p^{\ell}\Z) = p^{r_{p^\ell}(H) + r_{p^{\ell-1}}(H) + \cdots + r_p(H)} =
p^{\lambda_1'+\lambda_2' + \cdots + \lambda_\ell'} = T_{\ell}(H).
\]
\end{lemma}

\begin{proof}
Suppose
\[
H \cong \Z/p^{\lambda_1}\Z \times \cdots \times \Z/p^{\lambda_{r_p(H)}}\Z,
\]
and consider the particular generating set for $H$
\[
e_1 = (1,0,\ldots,0), e_2 = (0,1,0,\ldots,0),\ldots, e_{r_p(H)} = (0,\ldots, 0,1).
\]
Note that $e_i$ has order $p^{\lambda_i}$.

A homomorphism from $H$ to $\Z/p^\ell\Z$ is uniquely determined by the images of $e_1,\ldots, e_{r_p(H)}$.  When $\lambda_i \ge \ell$ there are $p^{\ell}$ choices for the image of $e_i$.  If $1 \le \lambda_i \le \ell$, there are $p^{\lambda_i}$ choices for the image of $e_i$.  Therefore, the total number of homomorphisms is
\[
p^{\ell  \lambda_\ell' + (\ell-1)  (\lambda_{\ell-1}' -\lambda_{\ell}') +  \cdots + 1 \cdot  (\lambda_1' - \lambda_2')}.
\]
\end{proof}

\begin{proof}[Proof of Theorem \ref{torspts}]
We  compute the expected value of \[ \#\Hom(H,\Z/p^{\ell}\Z) - \#\Hom(H,\Z/p^{\ell-1}\Z) \] and apply Lemma \ref{counthoms} to complete the proof.

Let $H$ be a finite abelian $p$-group drawn from $P_{d,u}$. Every element of $\Hom(H,\Z/p^{\ell}\Z)$ is either a surjection, or surjects onto a unique proper subgroup of $\Z/p^{\ell}\Z$.  Every proper subgroup of $\Z/p^{\ell}\Z$ is contained in the unique proper subgroup of $\Z/p^{\ell}\Z$ that is isomorphic to $\Z/p^{\ell-1}\Z$.  Therefore,
\[
\#\Surj(H,\Z/p^{\ell}\Z) = \#\Hom(H,\Z/p^{\ell}\Z) - \#\Hom(H,\Z/p^{\ell-1}\Z).
\]
Lemma \ref{counthoms} implies $T_{\ell}(H) - T_{\ell-1}(H) = \#\Surj(H,\Z/p^{\ell}\Z)$. Applying Theorem \ref{momgen}, noting that $T_{0}(H) = 1$ for any $H$, completes the proof.
\end{proof}

We close this section by proving a version of Lemma \ref{MomentsDetermine} for the distribution $P_{d,u}$.
The proof of Lemma 8.2 from \cite{EVW} carries over almost exactly to this more general setting.
\begin{theorem}\label{MomentsDetermine2}
Suppose that $p > 1$ and $0 < u < p$ are such that
\begin{equation}\label{Prod_pu}
\frac{1}{(u/p)_d} = \prod_{i=1}^{d} (1-u/p^i)^{-1} < 2.
\end{equation}
If $\nu$ is any probability measure on the set of partitions for which
\begin{equation}
\sum_{\lambda} \nu(\lambda)  |\Surj(\lambda,\mu)| =  \left\{ \begin{array}{ll}
\frac{u^{|\mu|} (1/p)_d}{(1/p)_{d-r(\mu)}} & \mbox{if $r(\mu) \leq d$} \\
0 & \mbox{otherwise},
\end{array} \right.
\end{equation}
then $\nu = P_{d,u}$.
\end{theorem}
{\it Remarks:}
\begin{itemize}
\item When $p$ is prime this result has an interpretation in terms of probability measures on $\LL$.

\item The exact same argument proves the analogous result for the distribution $P_{\infty, u}$.

\item The expression on the left-hand side of \eqref{Prod_pu} is decreasing in $p$ and in $u$.  Setting $d= \infty,\ u =1$ and noting that this inequality holds for all $p \ge 3$ gives Lemma \ref{MomentsDetermine}.

\item Similarly, setting $d = \infty,\ u = 1/p^w$ (with $p$ prime and $w$ a positive integer) gives Proposition 2.3 of \cite{W3}.

\item Theorem \ref{MomentsDetermine2} only applies when $1/(u/p)_d < 2$.  Results of Wood imply that the moments determine the distribution in additional cases where $p$ is prime, for example when $p = 2,\ d=\infty$, and $u =1$.  See Theorem 3.1 in \cite{W2} and Theorem 8.3 in \cite{W}.
\end{itemize}

\begin{proof}
The assumption gives, for every $\mu$
\begin{equation}\label{StartingEq}
|\Aut(\mu)|  \nu(\mu) + \sum_{\lambda \neq \mu } |\Surj(\lambda,\mu)|  \nu(\lambda) =  \left\{ \begin{array}{ll}
\frac{u^{|\mu|} (1/p)_d}{(1/p)_{d-r(\mu)}} & \mbox{if $r(\mu) \leq d$} \\
0 & \mbox{otherwise}.
\end{array} \right.
\end{equation}
Since the second term in the left-hand side of \eqref{StartingEq} is non-negative, for $r(\mu) > d$ we have $|\Aut(\mu)| \nu(\mu) = 0$, so $\nu(\mu) = 0$.

Now suppose that $r(\mu) \le d$.  Our goal is to show that
\[
\nu(\mu) = \frac{u^{|\mu|}(u/p)_d}{|\Aut(\mu)|} \frac{(1/p)_d}{(1/p)_{d-r(\mu)}}.
\]
By Theorem \ref{momgen}, in the particular case $\nu = P_{d,u}$, \eqref{StartingEq} is equal to
\[
 \frac{u^{|\mu|} (u/p)_d (1/p)_d}{(1/p)_{d-r(\mu)}} + \sum_{\substack{\lambda \neq \mu \\ r(\lambda) \le d}} u^{|\lambda|} (u/p)_d \frac{   |\Surj(\lambda,\mu)|}{|\Aut(\lambda)|}  \frac{(1/p)_d}{(1/p)_{d-r(\lambda)}} = \frac{u^{|\mu|} (1/p)_d}{(1/p)_{d-r(\mu)}} .
\]
This gives
\[
\sum_{\substack{\lambda \neq \mu \\ r(\lambda) \le d}} u^{|\lambda|} \frac{   |\Surj(\lambda,\mu)|}{|\Aut(\lambda)| (1/p)_{d-r(\lambda)}}
=
\frac{u^{|\mu |}}{(1/p)_{d-r(\mu)}}
\left(\frac{1}{(u/p)_d }  - 1\right).
\]

Let
\[
\beta = \frac{(1/p)_{d-r(\mu)}}{u^{|\mu|}} \sum_{\substack{\lambda \neq \mu \\ r(\lambda) \le d}} u^{|\lambda|} \frac{   |\Surj(\lambda,\mu)|}{|\Aut(\lambda)| (1/p)_{d-r(\lambda)}}
=
\frac{1}{(u/p)_d }  - 1.
\]
It is enough to show that
\begin{equation}\label{BetaEq}
|\Aut(\mu)| \nu(\mu) = u^{|\mu|}\frac{(1/p)_d}{(1/p)_{d-r(\mu)}}  \frac{1}{\beta+1}.
\end{equation}
By assumption, $|\beta| < 1$, so we verify \eqref{BetaEq} by showing that $|\Aut(\mu)| \nu(\mu)$ is bounded by the alternating partial sums of the series
\[
u^{|\mu|}\frac{(1/p)_d}{(1/p)_{d-r(\mu)}} \frac{1}{\beta+1} = u^{|\mu|}\frac{(1/p)_d}{(1/p)_{d-r(\mu)}} (1-\beta+\beta^2 - \cdots).
\]

Equation \eqref{StartingEq} implies that
\[
|\Aut(\mu)| \nu(\mu) \le \frac{u^{|\mu|} (1/p)_d}{(1/p)_{d-r(\mu)}}.
\]
For any $\lambda$ with $r(\lambda) \le d$ this gives
\[
\nu(\lambda) \le \frac{u^{|\lambda|} (1/p)_d}{|\Aut(\lambda)|(1/p)_{d-r(\lambda)}}.
\]
Using this bound in \eqref{StartingEq} gives
\begin{eqnarray*}
|\Aut(\mu)|  \nu(\mu) & = &  u^{|\mu|} \frac{(1/p)_d}{(1/p)_{d-r(\mu)}} -  \sum_{\substack{\lambda \neq \mu \\ r(\lambda) \le d}} |\Surj(\lambda,\mu)|  \nu(\lambda) \\
& \ge &   u^{|\mu|} \frac{(1/p)_d}{(1/p)_{d-r(\mu)}}- \sum_{\substack{\lambda \neq \mu \\ r(\lambda) \le d}} u^{|\lambda|} \frac{   |\Surj(\lambda,\mu)|}{|\Aut(\lambda)|} \frac{(1/p)_d}{(1/p)_{d-r(\lambda)}}\\
& & = \frac{u^{|\mu|} (1/p)_d}{(1/p)_{d-r(\mu)}}-  \frac{u^{|\mu|}(1/p)_d}{(1/p)_{d-r(\mu)}} \beta = \frac{u^{|\mu|} (1/p)_d}{(1/p)_{d-r(\mu)}}  (1-\beta).
\end{eqnarray*}

Similarly, for any $\lambda$ with $r(\lambda) \le d$, this gives
\[
\nu(\lambda) \ge \frac{u^{|\lambda|}}{|\Aut(\lambda)|} \frac{(1/p)_d}{(1/p)_{d-r(\lambda)}}  (1-\beta).
\]
Using this bound in \eqref{StartingEq} gives
\begin{eqnarray*}
|\Aut(\mu)|  \nu(\mu) & = &  u^{|\mu|} \frac{(1/p)_d}{(1/p)_{d-r(\mu)}} -  \sum_{\substack{\lambda \neq \mu \\ r(\lambda) \le d}} |\Surj(\lambda,\mu)|  \nu(\lambda) \\ \\
& \le &   u^{|\mu|} \frac{(1/p)_d}{(1/p)_{d-r(\mu)}} - \sum_{\substack{\lambda \neq \mu \\ r(\lambda) \le d}} u^{|\lambda|} \frac{   |\Surj(\lambda,\mu)|}{|\Aut(\lambda)|} \frac{(1/p)_d}{(1/p)_{d-r(\lambda)}}(1-\beta)
\end{eqnarray*}
which implies
\begin{eqnarray*}
|\Aut(\mu)|  \nu(\mu) &  \le &  u^{|\mu|} \frac{(1/p)_d}{(1/p)_{d-r(\mu)}}- u^{|\mu|} \frac{(1/p)_d}{(1/p)_{d-r(\mu)}}\beta (1-\beta) \\
& =  & u^{|\mu|} \frac{(1/p)_d}{(1/p)_{d-r(\mu)}} (1-\beta+\beta^2).
\end{eqnarray*}

Continuing in this way completes the proof.
\end{proof}

\end{document}